\documentclass[10pt]{article}
\usepackage{mathrsfs}
\usepackage{amsfonts}
\usepackage{amsthm,amsmath,amssymb,anysize}
\usepackage{amscd}
\newtheorem{Theorem}{Theorem}[section]
\newtheorem{Proposition}[Theorem]{Proposition}

\newtheorem{Lemma}[Theorem]{Lemma}
\newtheorem{Remark}[Theorem]{Remark}
\theoremstyle{definition}
\newtheorem{Definition}{Definition}[section]
\usepackage[numbers,sort&compress]{natbib}

\marginsize{3.5cm}{3.5cm}{3.6cm}{3cm}%×óÓÒÉÏÏÂ
\numberwithin{equation}{section}
\title{\textsf{Some properties of generalized reduced Verma modules over $\mathbb{Z}$-graded modular Lie
superalgebras}}
\author{\textsc{Keli Zheng$^{1,2}$} \;\;  \textsc{Yongzheng Zhang$^{2}$}\thanks{Corresponding author.}
 \\
  \\
  \textit{$^1$Department of Mathematics, Northeast Forestry University}\\
  \textit{Harbin 150040, P.R. China}\\
  \textit{$^2$School of Mathematics and Statistics},
  \textit{Northeast Normal University}\\
  \textit{Changchun 130024, P.R. China.}
  }

\date{ }
\begin{document}
\maketitle
\begin{abstract}
This paper is primarily concerned with generalized reduced Verma
modules over $\mathbb{Z}$-graded modular Lie superalgebras. Some
properties of the generalized reduced Verma modules and the
coinduced modules are obtained. Moreover, the invariant forms on the
generalized reduced Verma modules are considered. In particular, we
prove that the generalized reduced Verma module is isomorphic to the
mixed product for modules of $\mathbb{Z}$-graded modular Lie
superalgebras of Cartan type.
\end{abstract}

\footnote{E-mail addresses: zhengkl561@nenu.edu.cn (K. Zheng),
  zhyz@nenu.edu.cn (Y. Zhang).}
\section{Introduction}
\label{} As is well known, the representation theory plays an
important role in the research of Lie algebras and Lie superalgebras
(see \cite{ross,H,1,M} for examples). The question about the
structure of submodules of a Verma module arose in the original
paper of Verma \cite{Verma}. As a natural generalization of Verma
modules, the generalized Verma modules are modules induced, starting
from arbitrary simple modules (not necessarily finite-dimensional),
from a parabolic subalgebra and a complex semisimple Lie algebra
(see \cite{V,V1,xin,su2}). One of the main questions about
generalized Verma modules is their structure, i.e., reducibility,
submodules, equivalence, etc. The theory of generalized Verma
modules is rather similar to that of Verma modules. Some results of
Verma modules (see \cite{BGG,D}) were extended to certain class of
generalized Verma modules in \cite{R,FM,MS1,KM1,MO} (see also
references therein). But only rather particular classes of
generalized Verma modules were covered and the problem of how to say
something in a general case remains open.

The generalized reduced Verma module over modular Lie algebras was
constructed in \cite{Farnsteiner}. Some properties of generalized
reduced Verma module over modular Lie algebras were obtained (see
\cite{Farnsteiner, Farnsteiner1,qiusen}). Since generalized reduced
Verma modules are closely related to mixed products of modules, the
structure of mixed products seems to be important and interesting.
In \cite{shen1,shen2,shen3}, Shen classified the $\mathbb{Z}$-graded
irreducible representation of the $\mathbb{Z}$-graded Lie algebras
of Cartan type. His approach rests on the notion of the mixed
product. In \cite{qiusen} the graded modules of graded Cartan type
Lie algebras which possess nondegenerate invariant form were
determined by Chiu. In the case of modular Lie superalgebras of
Cartan type, $\mathbb{Z}$-graded modules of the $\mathbb{Z}$-graded
Lie superalgebras $W(n)$, $S(n)$ and $H(n)$, mixed products of
modules of infinite-dimensional Lie superalgebras and
$\mathbb{Z}$-graded modules of finite-dimensional Hamiltonian Lie
superalgebras were obtained in \cite{zhang1,zhang2,zhang3,zhang4},
respectively.

The aim of this paper is to partially generalize some beautiful
results about generalized reduced Verma modules over modular Lie
algebras in \cite{Farnsteiner, Farnsteiner1,qiusen}. In Section 2,
we review some necessary notions. In Section 3, some relations
between generalized reduced Verma modules and coinduced modules are
given. In Section 4, the invariant forms on generalized reduced
Verma modules are considered. In Section 5, we prove that
generalized reduced Verma modules are isomorphic to mixed products
for modules of $\mathbb{Z}$-graded modular Lie superalgebras of
Cartan type.

All Lie superalgebras and modules treated in the present paper are
assumed to be finite dimensional. In \cite{M,zhang} the reader could
find all notations and notions of Lie superalgebras and modular
representations which are not precisely defined in this paper.

\section{Preliminaries}\label{}

Throughout this paper we will assume that $\mathbb{F}$ is a field of
prime characteristic and $\mathbb{Z}_{2}=\{\bar{0},\bar{1}\}$ is the
residue class ring mod $2$. Let $L=L_{\bar{0}}\oplus L_{\bar{1}}$ be
a Lie superalgebra over $\mathbb{F}$. Then $\mathbb{F}$ has a
trivial structure of a $\mathbb{Z}_{2}$-graded $L$-module:
$\mathbb{F}_{\bar{0}}=\mathbb{F}$, $\mathbb{F}_{\bar{1}}=0$.
Furthermore, we always assumed that the representation of $L$ in
$\mathbb{F}$ is equal to zero.

In addition to the standard notation $\mathbb{Z}$, we write
$\mathbb{N}$ and $\mathbb{N}_{0}$ for the set of positive integers
and the set of nonnegative integers, respectively. Denote by
$\mathbb{N}_{0}^{k}$ the $k$-tuples with nonnegative integers as
entries. For any Lie superalgebra $L$ over $\mathbb{F}$, let $U(L)$
denote the universal enveloping algebra of $L$. If $L=\oplus_{i\in
\mathbb{Z}}L_{i}$ is a $\mathbb{Z}$-graded Lie superalgebra over
$\mathbb{F}$, we customarily put $L^{+}=\oplus_{i> 0}L_{i}$ and
$L^{-}=\oplus_{i< 0}L_{i}$. Then $L=L^{+}\oplus L_{0}\oplus L^{-}$
and $U(L)=U(L^{+})U(L_{0})U(L^{-})$.

Without being mentioned explicitly, if $d(x)$ ($zd(x)$) occurs in
some expression in this paper, we always regard $x$ as a
$\mathbb{Z}_{2}$-homogeneous ($\mathbb{Z}$-homogeneous) element and
$d(x)$ ($zd(x)$) as the $\mathbb{Z}_{2}$-degree
($\mathbb{Z}$-degree) of $x$.

\begin{Definition}[\cite{WZ}]
Let $V$ and $W$ be $L$-modules and suppose that $f$ is a $\mathbb{Z}_{2}$-homogeneous element of
$\mathrm{Hom}_{\mathbb{F}}(V,W)$. The mapping $f$ is called a
\emph{homomorphism} of $L$-modules if $(x\cdot
f)(v)=(-1)^{d(x)d(f)}f(x\cdot v)$ for all $x\in L$ and $v\in V$. The
mapping $f$ is said to be an \emph{isomorphism} of $L$-modules if
$f$ is an homomorphism and if, furthermore, $f$ is a bijection.
\end{Definition}

Let $V$ be an $L$-module. The vector space
$V^{*}:=\mathrm{Hom}_{\mathbb{F}}(V,\mathbb{F})$ obtains the
structure of an $L$-module by means of $(x\cdot
f)(v)=-(-1)^{d(x)d(f)}f(x\cdot v)$, where $x\in L$, $v\in V$, $f\in
V^{*}$. Clearly, $d(x\cdot f)=d(x)+d(f)$.

We consider the subalgebra $K:=L_{0}\oplus L^{+}$ of a
$\mathbb{Z}$-graded Lie superalgebra $L=\oplus_{i\in
\mathbb{Z}}L_{i}$. Let $\{e_{1},\ldots,e_{k}\}$ be a basis of
$L^{-}\cap L_{\bar{0}}$ and $\{\xi_{1},\ldots,\xi_{l}\}$ be a basis
of $L^{-}\cap L_{\bar{1}}$. As $L^{-}\cap L_{\bar{0}}$ operates on
$L$ by nilpotent transformation, there exist $m_{i}\in
\mathbb{N}_{0}$, $1\leq i\leq k$ such that
\[
z_{i}:=e_{i}^{p^{m_{i}}}\in U(L^{-})\cap Z(U(L)), \quad 1\leq i\leq
k,
\]
where $Z(U(L))$ is the center of $U(L)$. In particular, $\{z_{i}\mid
1\leq i\leq k \}$ are homogeneous elements relative to the
$\mathbb{Z}$-gradation inherited by $U(L_{\bar{0}})$. An application
of P-B-W Theorem (see \cite{ross}) reveals that the subalgebra
$\theta(L,K)$ of $U(L)$, which is generated by $K$ and
$\{z_{1},\ldots,z_{k}\}$, is isomorphic to
$\mathbb{F}[z_{1},\ldots,z_{k}]\otimes_{\mathbb{F}}U(K)$, where
$\mathbb{F}[z_{1},\ldots,z_{k}]$ is a polynomial ring in $k$
indeterminates. Then an easy computation shows that $\theta(L,K)$ is
a $\mathbb{Z}$-graded subalgebra of $U(L)$.

Given $\alpha=(\alpha_{1},\ldots,\alpha_{k})\in\mathbb{N}_0^{k}$, we
put $|\alpha|:=\sum_{i=1}^{m}\alpha_{i}$,
$e^{\alpha}:=e_{1}^{\alpha_{1}}e_{2}^{\alpha_{2}}\cdots
e_{k}^{\alpha_{k}}$ and
$\pi:=(\pi_{1},\ldots,\pi_{k})=(p^{m_{1}}-1,\ldots,p^{m_{k}}-1)$.
Set
\[
\mathbb{B}_{s}:=\left\{\langle i_{1},i_{2},\ldots,i_{s}\rangle\mid
1\leq i_{1}< i_{2}<\cdots< i_{s}\leq l\right\}
\]
and $\mathbb{B}:=\bigcup_{s=0}^{l}\mathbb{B}_{s}$, where
$\mathbb{B}_{0}:=\emptyset$ and $l\in \mathbb{N}$. For $u=\langle
i_{1},i_{2},\ldots,i_{s}\rangle\in \mathbb{B}_{s}$, set $|u|:=s$,
$|\emptyset|:=0$, $\xi^{\emptyset}:=1$,
$\xi^{u}:=\xi_{i_{1}}\xi_{i_{2}}\cdots \xi_{i_{s}}$ and
$\xi^{E}:=\xi_{1}\xi_{2}\cdots \xi_{l}$, $u$ is also used to stand
for the index set $\{ i_{1},i_{2},\ldots,i_{s}\}$. It is easy to
show that $U(L)$ is a $\mathbb{Z}$-graded $\theta(L,K)$-module with
the basis
\[
\{e^{\alpha}\xi^{u}\mid 0\leq \alpha\leq \pi, u\in \mathbb{B}\}.
\]

Any $K$-module $V$ obtains the structure of a $\theta(L,K)$-module
by letting $\mathbb{F}[z_{1},\ldots,z_{k}]$ act via its canonical
supplementation which sends $z_{i}$ to $0$. Henceforth, $K$-module will be regarded as
$\theta(L,K)$-module in this fashion. Let $\rho$ be the natural
representation of $K$ in $L/K$. Then there exists a unique
homomorphism $\sigma:U(K)\rightarrow \mathbb{F}$ of
$\mathbb{F}$-superalgebra such that
$\sigma(x)=\mathrm{str}(\rho(x))$, where $x$ is an arbitrary element
of $K$ and $\mathrm{str}(\rho(x))$ is the supertrace of $\rho(x)$
(see \cite{1,M}). We introduce a twisted action on $K$-module $V$ by
setting
\[
x\circ v=x\cdot v+\sigma(x)v, \quad x\in K, \quad v\in V.
\]
Note that $\sigma(x)=0$ for $x\in K_{\bar{1}}$, then
\begin{eqnarray*}
[x,y]\circ v&=&[x,y]\cdot v+\sigma([x,y])v\\
&=&x\cdot (y\cdot v)-(-1)^{d(x)d(y)}y\cdot (x\cdot v)+\sigma(x)\sigma(y)v-(-1)^{d(x)d(y)}\sigma(y)\sigma(x)v\\
&=&x\cdot (y\cdot v)+\sigma(y)x\cdot v+\sigma(x)y\cdot v+\sigma(x)\sigma(y)v\\
& &-(-1)^{d(x)d(y)}y\cdot (x\cdot v)-(-1)^{d(x)d(y)}\sigma(y)x\cdot v\\
& &-(-1)^{d(x)d(y)}\sigma(x)y\cdot v-(-1)^{d(x)d(y)}\sigma(y)\sigma(x)v\\
&=&x\cdot(y\circ v)+\sigma(y)(x\circ v)-(-1)^{d(x)d(y)}y\cdot(x\circ v)-(-1)^{d(x)d(y)}\sigma(y)(x\circ v)\\
&=&x\circ (y\circ v)-(-1)^{d(x)d(y)}y\circ (x\circ v),
\end{eqnarray*}
i.e., $V$ is a new $K$-module by the twisted action.
The new $K$-module will be denoted by $V_{\sigma}$.
If $V$ is an $L_{0}$-module, then we can extend the operations on $V$ to $K$ by
letting $L^{+}$ act trivially and regard it as a $K$-module.

\section{Generalized reduced Verma modules and coinduced modules}
Let $L$ be a $\mathbb{Z}$-graded Lie superalgebra over $\mathbb{F}$
and $V$ be a $K$-module. Following \cite{Farnsteiner}, we give the
definition as follow.

\begin{Definition}
The induced module $\mathrm{Ind}_{K}(V):=U(L)\otimes_{\theta(L,K)}V$
is called a \emph{generalized reduced Verma module}. The coinduced
module $\mathrm{Hom}_{\theta(L,K)}(U(L),V)$ will be denoted by
$\mathrm{Coind}_{K}(V)$.
\end{Definition}

It is clear from the above construction that the modules
$\mathrm{Ind}_{K}(V)$ and $\mathrm{Coind}_{K}(V)$ are annihilated by
$z_{i}$.

Consider $\mathrm{Coind}_{K}(V)$ with $U(L)$-action given via
\[
(y\cdot f)(x):=(-1)^{d(y)(d(f)+d(x))}f(xy), \quad x,y\in U(L).
\]
For $v\in V$, %$0\leq\alpha \leq \pi$,
$0\leq\beta \leq \pi$ and
$u,t\in \mathbb{B}$, let $\chi_{v}^{(\beta,t)}$ be the element of
$\mathrm{Coind}_{K}(V)$ which sends $e^{\alpha}\xi^{u}$ onto
$(-1)^{d(\chi_{v}^{(\beta,t)})d(\xi^{u})}\delta(\alpha,\beta)\delta(u,t)v$
, where $\delta(i,j)$ is Kronecker delta, defined by $\delta(i,j)=1$
if $i=j$ and $\delta(i,j)=0$ otherwise. It obviously suffices to
verify that
\begin{equation}
\chi_{v}^{(\beta,t)}(e^{\beta}\xi^{t}\vartheta)=(-1)^{d(\vartheta)(
d(\chi_{v}^{(\beta,t)})+d(\xi^{t}))+d(\chi_{v}^{(\beta,t)})d(\xi^{t})}\vartheta\circ v\label{eq1}
\end{equation}
and $d(\chi_{v}^{(\beta,t)})=d(\xi^{t})+d(v)$, for all $\vartheta\in
\theta(L,K)$, for all $v\in V_{\sigma}$.

\begin{Lemma}\label{lm1}
There is a natural isomorphism of functors
\[\Phi:\mathrm{Ind}_{K}(V_{\sigma})\rightarrow \mathrm{Coind}_{K}(V)\]
such that $\Phi(y\otimes v)=(-1)^{d(y)d(\Phi)}y\cdot
\chi_{v}^{(\pi,E)}$, where $y\in U(L)$ and $v\in V_{\sigma}$.
\end{Lemma}

\begin{proof}
Suppose that the bilinear mapping $\psi: U(L)\times
V_{\sigma}\rightarrow \mathrm{Hom}_{\mathbb{F}}(U(L),V)$ is defined
by $\psi(y,v)=(-1)^{d(y)d(\psi)}y\cdot \chi_{v}^{(\pi,E)}$. Let
$\vartheta\in \theta(L,K)$ and $u'\in U(L)$. Then the equation
(\ref{eq1}) and $d(\chi_{v}^{(\pi,E)})=d(\psi)+d(v)$ imply that
\begin{eqnarray*}
\psi(y\vartheta,v)(u')&=&(-1)^{(d(y)+d(\vartheta))d(\psi)}y\vartheta\cdot \chi_{v}^{(\pi,E)}(u')\\
&=&(-1)^{(d(y)+d(\vartheta))(d(v)+d(u'))}\chi_{v}^{(\pi,E)}(u'y\vartheta)\\
&=&(-1)^{d(y)(d(v)+d(\vartheta)+d(u'))+d(\vartheta)d(\psi)+(d(\psi)+d(v))(d(u')+d(y))}\vartheta\circ v\\
&=&(-1)^{d(y)(d(v)+d(\vartheta)+d(u'))+(d(\vartheta)+d(\psi)+d(v))(d(u')+d(y))}\vartheta\circ v\\
&=&(-1)^{d(y)(d(v)+d(\vartheta)+d(u'))}\chi_{\vartheta\circ v}^{(\pi,E)}(u'y)\\
&=&(-1)^{d(y)d(\psi)}y\cdot \chi_{\vartheta\circ v}^{(\pi,E)}(u')\\
&=&\psi(y,\vartheta\circ v)(u').
\end{eqnarray*}
Consequently, $\psi$ is $\theta(L,K)$-balanced, and induces a
mapping
\[\Phi:U(L)\otimes_{\theta(L,K)} V_{\sigma}\rightarrow \mathrm{Hom}_{\mathbb{F}}(U(L),V).\]
The verification of the inclusion $\mathrm{im}\psi\subseteq
\mathrm{Hom}_{\theta(L,K)}(U(L),V)$ is routine.

For any $x,y\in U(L)$ and $v\in V_{\sigma}$, we have
\[
(x\cdot \Phi)(y\otimes v)=(-1)^{d(y)d(\Phi)}((xy)\cdot \chi_{v}^{(\pi,E)})=(-1)^{d(x)d(\Phi)}\Phi(x\cdot(y\otimes v)).
\]
Hence $\Phi$ is a homomorphism of $U(L)$-modules.

For any $f\in \mathrm{Coind}_{K}(V)$, there exists
$e^{\alpha}\xi^{u}\in U(L)$ such that
\[f=\sum_{\alpha,u}(-1)^{d(f)d(\xi^{u})}\chi_{f(e^{\alpha}\xi^{u})}^{(\alpha,u)},\]
where $0\leq \alpha\leq \pi$ and $u\in \mathbb{B}$. Then
$\Phi(\sum_{\alpha,u}(-1)^{d(f)d(\xi^{u})}y\otimes
f(e^{\alpha}\xi^{u}))=f$, i.e., $\Phi$ is a surjection.

Suppose that $0=y\cdot X_{v}^{(\pi,E)}\in \mathrm{Coind}_{K}(V)$ and
$y=e^{\alpha}\xi^{u}\in U(L)$, then there exists
$u'=e^{\beta}\xi^{t}\in U(L)$ such that $\alpha+\beta=\pi$ and
$u+t=E$. It follows that \[0=y\cdot
\chi_{v}^{(\pi,E)}(u')=(-1)^{d(y)(d(u')+d(\chi_{v}^{(\pi,E)}))+d(\chi_{v}^{(\pi,E)})(d(u')+d(y))}v.\]
Therefore, $y\otimes v=0$, i.e., $\Phi$ is an injection.

Now we show that $\Phi$ is a natural homomorphism. Suppose that $W$ is a $K$-module and $\varphi:V\rightarrow W$ is a homomorphism of $K$-module.
Clearly, $\varphi$ is also a homomorphism between $V_{\sigma}$ and $W_{\sigma}$. We claim that the following diagram is commutative.
\[ \begin{CD}
  \mathrm{Ind}_{K}(V_{\sigma}) @>\Phi>> \mathrm{Coind}_{K}(V) \\
  @V\mathrm{id}\otimes \varphi VV                       @VV\varphi^{*}V  \\
  \mathrm{Ind}_{K}(W_{\sigma}) @>\Phi'>> \mathrm{Coind}_{K}(W)
\end{CD} \]
Note that $\varphi^{*}$ and $\mathrm{id}\otimes \varphi$ are homomorphisms of $U(L)$-modules,
the assertion follows from the ensuing calculation:
\[
\varphi^{*}\circ \Phi(1\otimes
v)(u')=\chi_{\varphi(v)}^{(\pi,E)}(u')=(\Phi'\circ(\mathrm{id}\otimes
\varphi))(1\otimes v)(u'),\quad u'\in U(L).
\]
In conclusion, the proof is completed.
\end{proof}

\begin{Remark}
\begin{enumerate}
\item[(1)] If the above result is applied to the module $V_{-\sigma}$, then we obtain
natural isomorphisms
$\mathrm{Ind}_{K}(V)\cong\mathrm{Coind}_{K}(V_{-\sigma})$.
\item[(2)] Suppose that $K$ acts nilpotently on $L/K$ or $(\rho(K))^{(1)}=\rho(K)$. Then $\sigma=0$ and every $K$-module $V$ gives an isomorphism
$\mathrm{Ind}_{K}(V)\cong\mathrm{Coind}_{K}(V)$.
\end{enumerate}
\end{Remark}

Following \cite{shen2}, we refer to a $\mathbb{Z}$-graded $L$-module
$V$ as positively graded if $V=\bigoplus\limits_{i\geq 0}V_{i}$ and
$L_{j}\cdot V_{i}\subseteq V_{i+j}$. A positively graded module $V$
is said to be transitive if $V_{0}=\{v\in V\mid x\cdot v=0$, for all
$x\in L^{-}\}$.

\begin{Proposition}\label{p1}
Let $P=\mathrm{Coind}_{K}(V)$ be an $L$-module and \[P_{i}:=\{f\in
P\mid f(U(L)_{j})=0, j\neq -i\}.\] Then
\begin{enumerate}
\item[(1)] $P$ is a positively graded $L$-module.
\item[(2)] $P_{0}$ is isomorphism to $V$ as an $L_{0}$-module.
\item[(3)] $P$ is transitively graded.
\end{enumerate}
\end{Proposition}

\begin{proof}
\begin{enumerate}
\item[(1)] Let $f$ be an element of $P_{i}$  and suppose that $y\in U(L)_{q}$, where $i,q\in \mathbb{Z}$.
If $x\in U(L)_{j}$ for $j\neq -i-q$, then $xy\in U(L)_{j+q}$, where
$j\in \mathbb{Z}$. It follows that \[(y\cdot
f)(x)=(-1)^{d(y)(d(f)+d(x))}f(xy)=0.\] Consequently, $(y\cdot f)$
belongs to $P_{i+q}$.

Let $\{x_{1},\ldots,x_{n}\}$ be the basis of $U(L)$ over
$\theta(L,K)$ and induced by $\{e_{1},\ldots,e_{k}\}$ and
$\{\xi_{1},\ldots,\xi_{l}\}$. In accordance with the basis of
$U(L)$, we may assume that $x_{r}=e^{\alpha}\xi^{u}\in U(L)_{i(r)}$,
where $i(r)\leq 0$ and $1\leq r\leq n$. Any element of $U(L)_{q}$ is
a sum of elements $x=\sum_{r=1}^{n}h_{r}x_{r}$, $h_{r}\in
\theta(L,K)_{q-i(r)}$. Given $r\in \{1,2,\ldots,n\}$, we have
$\chi_{v}^{(\alpha,u)}(x)=(-1)^{(d(x)+d(v))d(x)}h_{r}v$. If $q\neq
i(r)$, then $\chi_{v}^{(\alpha,u)}(x)=0$. It follows that
$\chi_{v}^{(\alpha,u)}$ is an element of $P_{-i(r)}$. For every
$f\in P$, we have
$f=\sum_{\alpha,u}(-1)^{d(f)d(\xi^{u})}\chi_{f(e^{\alpha}\xi^{u})}^{(\alpha,u)}$.
Consequently, $P=\oplus_{r=1}^{n} P_{-i(r)}$ and $P$ is a positively
graded module.

\item[(2)] We proceed by showing that $\mu:P_{0}\rightarrow V$; $\mu(f)=f(1)$ is an isomorphism of $L_{0}$-modules.
If $x\in L_{0}$, then
\[
\mu(x\cdot f)=(x\cdot f)(1)=(-1)^{d(x)d(f)}f(x)=x\cdot f(1)=x\cdot\mu(f),
\]
i.e., $\mu$ is a homomorphism of $L_{0}$-modules.

Since $1:=e^{\alpha}\xi^{u}\in U(L)_{0}$ is contained in
$\{x_{1},\ldots,x_{n}\}$,
$(-1)^{(d(\xi^{u})+d(v))d(\xi^{u})}\chi_{v}^{(\alpha,u)}$ is a
pre-image of $v\in V$ under $\mu$.

Suppose that $f\in \mathrm{ker}\mu$.
Owing to the P-B-W theorem, for every element $x\in U(L)_{0}$, we may assume that
$x=\sum_{i+j=0}a_{i}b_{j}$, where $a_{i}\in U(K)_{i}$ and $b_{j}\in U(L^{-})_{j}$.
Since $a_{i}=0$ for $i<0$ and $a_{i}\in U(L_{0})U(L^{+})$ for $i>0$,
we obtain
\begin{eqnarray*}
f(x)&=&\sum_{i+j=0}(-1)^{d(a_{i})d(f)}a_{i}f(b_{j})=(-1)^{d(a_{0})d(f)}a_{0}f(b_{0})\\
&=&(-1)^{(d(a_{0})+d(b_{0}))d(f)}a_{0}b_{0}f(1)=0.
\end{eqnarray*} As a result $f=0$ on $U(L_{0})$ and thereby on all of
$U(L)$. Therefore, $\mu$ is an isomorphism of $L_{0}$-modules.

\item[(3)]  Suppose that $f$ is an element of $P$ such that $x\cdot f=0$ for every $x\in L^{-}$.
Then each $\mathbb{Z}$-homogeneous constituent of $f$ enjoys the
same property. Thus we may assume $f\in P_{q}$, where $q\in
\mathbb{Z}$. Suppose that $q>0$ and $y$ is an element of
$U(L)_{-q}$. Without loss of generality we may assume that
$y=\sum_{i+j=-q}a_{i}b_{j}$, where $a_{i}\in U(K)_{i}$ and $b_{j}\in
U(L^{-})_{j}$. As $a_{i}\cdot V=0$ for $i>0$, we have
\[
f(y)=\sum_{i+j=-q}(-1)^{d(a_{i})d(f)}a_{i}f(b_{j})=(-1)^{d(a_{0})d(f)}a_{0}f(b_{-q}).
\]
Then it follows that
$f(y)=(-1)^{(d(a_{0})+d(b_{-q}))d(f)}a_{0}b_{-q}f(1)$. Since
$b_{-q}$ belongs to $U(L^{-})$, we obtain $b_{-q}\cdot f=0$. Thus
$f(y)=0$. Similarly, if $q<0$, then $f(y)$ also equals to zero.
Therefore, $f\in P_{0}$.

Conversely, if $f\in P_{0}$, then $f(U(L)_{i})=0$ for $i\neq 0$. For
any $x\in L^{-}$, we have \[ (x\cdot
f)(y)=(-1)^{d(x)(d(f)+d(y))}f(yx)=(-1)^{d(x)(d(y))}y\cdot f(x)=0,
\quad y\in U(L)^{+}
\]
and
\[
(x\cdot f)(y)=(-1)^{d(x)(d(f)+d(y))}f(yx)=0, \quad y\in
U(L)^{-}\oplus U(L)_{0}.
\]
Therefore, $x\cdot f=0$ for all $x\in L^{-}$.
\end{enumerate}
\end{proof}

For $x_{1},\ldots,x_{n}\in L$,
set
\[
(x_{1}\cdots x_{n})^{T}:=(-1)^{n+\sum_{i=1}^{n-1}\sum_{j=i+1}^{n}d(x_{i})d(x_{j})}x_{n}\cdots x_{1}.
\]
It is easy to verify that $x_{i}^{T}=-x_{i}$ and
$d(x_{i}^{T})=d(x_{i})$ for $i\in \{1,\ldots,n\}$. Then the
principal anti-automorphism of $U(L)$ is defined by $x\mapsto
x^{T}$, for all $x\in U(L)$.

In the following proposition, the property of adjoint isomorphism
will be investigated.

\begin{Proposition}\label{p2}
There is a natural isomorphism:
\[\Psi:(\mathrm{Ind}_{K}(V))^{*}\rightarrow \mathrm{Coind}_{K}(V^{*}),\]
namely, for $\varphi\in (\mathrm{Ind}_{K}(V))^{*}$, $x\in U(L)$ and $v\in V$,
\[\Psi:\varphi\mapsto\Psi(\varphi), \mbox{ where }  \Psi(\varphi)(x):v\mapsto \varphi(x^{T}\otimes v).\]
\end{Proposition}

\begin{proof}
Firstly, we prove that $\Psi$ is a homomorphism of $U(L)$-modules.
Let $\varphi_{1}$ and $\varphi_{2}$ are elements of
$(\mathrm{Ind}_{K}(V))^{*}$. Then the definition of
$\varphi_{1}+\varphi_{2}$ shows that
\begin{eqnarray*}
\Psi(\varphi_{1}+\varphi_{2})(x)(v)&=&(\varphi_{1}+\varphi_{2})(x^{T}\otimes v)\\
&=&(\varphi_{1})(x^{T}\otimes v)+(\varphi_{2})(x^{T}\otimes v)\\
&=&\Psi(\varphi_{1})(x)(v)+\Psi(\varphi_{2})(x)(v)\\
&=&(\Psi(\varphi_{1})+\Psi(\varphi_{2}))(x)(v),
\end{eqnarray*}
where $x\in U(L)$ and $v\in V$.
Therefore, $\Psi(\varphi_{1}+\varphi_{2})=\Psi(\varphi_{1})+\Psi(\varphi_{2})$.
For any $x,y\in U(L)$,  $v\in V$ and $\varphi\in (\mathrm{Ind}_{K}(V))^{*}$, we have
\begin{eqnarray*}
y\cdot \Psi(\varphi)(x)(v)&=&(-1)^{d(y)(d(\Psi)+d(\varphi)+d(x))}\Psi(\varphi)(xy)(v)\\
&=&(-1)^{d(y)(d(\Psi)+d(\varphi)+d(x))}\varphi((xy)^{T}\otimes v)\\
&=&(-1)^{d(y)(d(\Psi)+d(\varphi))}\varphi(yx\otimes v)\\
&=&(-1)^{d(y)d(\Psi)}y\cdot\varphi(x^{T}\otimes v)\\
&=&(-1)^{d(y)d(\Psi)}\Psi(y\cdot\varphi)(x)(v).
\end{eqnarray*}
Therefore, $y\cdot \Psi(\varphi)=(-1)^{d(y)d(\Psi)}\Psi(y\cdot\varphi)$.

Next $\Psi$ is injective. In fact, if $\Psi(\varphi)(x)(v)=0$, then
$0=\Psi(\varphi)(x)(v)=\varphi(x^{T}\otimes v)$ for all $x\in U(L)$
and $v\in V$. Thus $\varphi=0$ because it vanishes on every
generator of $\mathrm{Ind}_{K}(V)$.

Now we show that $\Psi$ is surjective.
Let $f\in \mathrm{Coind}_{K}(V^{*})$. Define $\varphi(x\otimes v):=f(x^{T})(v)$ for $x\in U(L)$ and $v\in V$.
Then $\Psi(\varphi)=f$.

It is easy to check that $\Psi$ is a natural homomorphism. In
conclusion, the proof is completed.
\end{proof}

\begin{Remark}
Proposition \ref{p2} is called adjoint isomorphism in homological algebra (see \cite{rotman}).
\end{Remark}

\begin{Theorem}\label{th1}
$\mathrm{Ind}_{K}(V_{\sigma})\cong(\mathrm{Ind}_{K}(V_{\sigma}))^{*}$ if and only if $V\cong (V_{\sigma})^{*}$.
\end{Theorem}

\begin{proof}
If
$\mathrm{Ind}_{K}(V_{\sigma})\cong(\mathrm{Ind}_{K}(V_{\sigma}))^{*}$,
by Lemma \ref{lm1} and Proposition \ref{p2}, then
\begin{eqnarray*}
\mathrm{Coind}_{K}(V)\cong \mathrm{Coind}_{K}((V_{\sigma})^{*}).
\end{eqnarray*}
It follows from Proposition \ref{p1} that $V\cong (V_{\sigma})^{*}$.
The sufficiency is obvious.
\end{proof}

\section{Invariant forms on generalized reduced Verma modules}
\label{} The results of this section generalize Chiu's results in
\cite{qiusen} and determine generalized reduced Verma modules over
modular Lie superalgebras which possess a nondegenerate
super-symmetric or skew super-symmetric invariant bilinear form. Let
$L$ be a Lie superalgebra over $\mathbb{F}$ and $V$ be an
$L$-module. A bilinear form $\lambda:V\times V\rightarrow
\mathbb{F}$ is called super-symmetric (skew super-symmetric) if
$\lambda(v,w)=(-1)^{d(v)d(w)}\lambda(w,v)$
($\lambda(v,w)=-(-1)^{d(v)d(w)}\lambda(w,v)$), for all $v,w\in V$. A
super-symmetric (or skew super-symmetric) bilinear form
$\lambda:V\times V\rightarrow \mathbb{F}$ is called invariant on $L$
if $\lambda(x\cdot v,w)=-(-1)^{d(v)d(x)}\lambda(v,x\cdot w)$, for
all $x\in L$ and $v,w\in V$. The subspace
$\mathrm{rad}(\lambda):=\{v\in V\mid \lambda(v,w)=0$, for all $w\in
V\}$ is called the radical of $\lambda$. The form $\lambda$ is
nondegenerate if $\mathrm{rad}(\lambda)=0$.

\begin{Proposition}\label{p3}
There exists a nondegenerate super-symmetric (skew super-symmetric)
invariant bilinear form $\lambda$ on $V$ if and only if there exists
an isomorphism of $L$-modules $\phi:V\rightarrow V^{*}$ such that
$\phi(v)(w)=(-1)^{d(v)d(w)}\phi(w)(v)$
($\phi(v)(w)=-(-1)^{d(v)d(w)}\phi(w)(v)$), for all $v,w\in V$.
\end{Proposition}

\begin{proof}
Let $\lambda$ be a nondegenerate super-symmetric (skew
super-symmetric) invariant bilinear form on $V$. Define
$\phi:V\rightarrow V^{*}$ such that $\phi(v)(w):=\lambda(v,w)$, for
all $v,w\in V$. Obviously, $\phi$ is a linear mapping such that
$\mathrm{ker}\phi=\mathrm{rad}(\lambda)=0$ and
$\phi(v)(w)=(-1)^{d(v)d(w)}\phi(w)(v)$
($\phi(v)(w)=-(-1)^{d(v)d(w)}\phi(w)(v)$). Hence $\phi$ is
injective. Since $\mathrm{dim}V=\mathrm{dim}V^{*}$, $\phi$ is
bijective. For $x\in L$ and $v,w\in V$, we have
\begin{eqnarray*}
\phi(x\cdot v)(w)&=&\lambda(x\cdot v,w)=-(-1)^{d(x)d(v)}\lambda(v,x\cdot w)\\
&=&-(-1)^{d(x)d(v)}\phi(v)(x\cdot w)=(-1)^{d(x)d(v)}(x\cdot \phi(v))(w).
\end{eqnarray*}
Thus $\phi$ is the desired isomorphism of $L$-modules.

Conversely, let $\phi$ be an isomorphism of $L$-modules such that
$\phi(v)(w)=(-1)^{d(v)d(w)}\phi(w)(v)$
($\phi(v)(w)=-(-1)^{d(v)d(w)}\phi(w)(v)$), for all $v,w\in V$. Put
$\lambda(v,w):=\phi(v)(w)$. Thus $\lambda$ be a super-symmetric
(skew super-symmetric) bilinear form on $V$. Furthermore,
\begin{eqnarray*}
\lambda(x\cdot v,w)&=&\phi(x\cdot v)(w)=(-1)^{d(x)d(\phi)}(x\cdot \phi(v))(w)\\
&=&-(-1)^{d(x)d(v)}\phi(v)(x\cdot w)=-(-1)^{d(x)d(v)}\lambda(v,x\cdot w),
\end{eqnarray*}
for all $x\in L$ and $v,w\in V$. Hence $\lambda$ is invariant. As
$\mathrm{rad}(\lambda)=\mathrm{ker}\phi=0$, $\lambda$ is
nondegenerate.
\end{proof}

\begin{Proposition}\label{p4}
Let $V$ be an irreducible $L$-module. If $V$ is isomorphic to $V^{*}$ as $L$-module,
then there exists a nondegenerate invariant bilinear form $\lambda$ on $V$ which is either super-symmetric or skew super-symmetric.
\end{Proposition}

\begin{proof}
By the proof of Proposition \ref{p3}, there exists a nondegenerate invariant bilinear form $\beta$ on $V$.
Let
\[
\lambda(v,w)=\beta(v,w)+(-1)^{d(v)d(w)}\beta(w,v), \quad v,w\in V.
\]
Clearly, $\lambda$ is a super-symmetric bilinear form on $V$. Since
$V$ is an irreducible $L$-module, this implies that
$\mathrm{rad}(\lambda)=\mathrm{ker}\phi=0$ or
$\mathrm{rad}(\lambda)=\mathrm{ker}\phi=V$. Therefore, $\lambda$ is
either nondegenerate or $0$. It follows that either $\beta$ or
$\lambda$ is desired form.
\end{proof}

\begin{Theorem}\label{th2}
Let $L$ be a $\mathbb{Z}$-graded Lie superalgebra over $\mathbb{F}$
and $V$ be an $L_{0}$-module. Then the following statements are
equivalent.
\begin{enumerate}
\item[(1)] There exists a nondegenerate super-symmetric or skew super-symmetric invariant bilinear form on $\mathrm{Ind}_{K}(V_{\sigma})$.
\item[(2)] There exists an isomorphism of $L_{0}$-modules
$\zeta:V\rightarrow (V_{\sigma})^{*}$ such that
$\zeta(v)(w)=(-1)^{d(v)d(w)}\zeta(w)(v)$ or
$\zeta(v)(w)=-(-1)^{d(v)d(w)}\zeta(w)(v)$, $v,w\in V$.
\end{enumerate}
\end{Theorem}

\begin{proof}
Suppose that there exists a nondegenerate super-symmetric or skew super-symmetric invariant bilinear form on $\mathrm{Ind}_{K}(V_{\sigma})$.
By Proposition \ref{p3}, there exists an isomorphism of $L$-modules
$\phi:\mathrm{Ind}_{K}(V_{\sigma})\rightarrow (\mathrm{Ind}_{K}(V_{\sigma}))^{*}$ such that
\begin{eqnarray}
\phi(x_{1}\otimes v_{1})(x_{2}\otimes v_{2})
&=&(-1)^{(d(x_{1})+d(v_{1}))(d(x_{2})+d(v_{2}))}\phi(x_{2}\otimes v_{2})(x_{1}\otimes v_{1})\label{1}
\end{eqnarray}
or
\begin{eqnarray}
\phi(x_{1}\otimes v_{1})(x_{2}\otimes v_{2})
&=&-(-1)^{(d(x_{1})+d(v_{1}))(d(x_{2})+d(v_{2}))}\phi(x_{2}\otimes v_{2})(x_{1}\otimes v_{1}),\label{2}
\end{eqnarray}
where $x_{1},x_{2}\in U(L)$ and $v_{1},v_{2}\in V$. Theorem
\ref{th1} shows that there exists an isomorphism of $L_{0}$-modules
$\zeta:V\rightarrow (V_{\sigma})^{*}$.

Let $x_{1}=e^{\alpha}\xi^{u}\in U(L^{-})$ and $x_{2}=e^{\beta}\xi^{t}\in U(L^{-})$,
where $0\leq\alpha\leq \pi$, $0\leq\beta\leq \pi$ and $u,t\in \mathbb{B}$.
By the proof of Lemma \ref{lm1} and Proposition \ref{p2}, we have
\begin{eqnarray}
& &\phi(x_{1}\otimes v_{1})(x_{2}\otimes v_{2})
=(-1)^{d(x_{1})d(x_{2})+d(x_{1})d(v_{1})}\chi_{\zeta(v_{1})}^{(\pi,E)}(x_{2}^{T}x_{1})(v_{2})\nonumber\\
&=&(-1)^{d(x_{1})d(x_{2})+d(x_{1})d(v_{1})+(d(\zeta)+d(v_{1})+d(\xi^{E}))(d(x_{1})+d(x_{2}))}\delta(\pi,\alpha+\beta)\delta(E,u+t)\zeta(v_{1})(v_{2})\nonumber\\
&=&(-1)^{d(x_{1})d(x_{2})+d(x_{2})d(v_{1})+(d(\zeta)+d(\xi^{E}))(d(x_{1})+d(x_{2}))}\zeta(v_{1})(v_{2}).\label{3}
\end{eqnarray}
According to (\ref{1}), (\ref{2}) and (\ref{3}),
\[\zeta(v_{1})(v_{2})=(-1)^{d(v_{1})d(v_{2})}\zeta(v_{2})(v_{1})
\mbox{  or }
\zeta(v_{1})(v_{2})=-(-1)^{d(v_{1})d(v_{2})}\zeta(v_{2})(v_{1}),\]
for all $v_{1},v_{2}\in V$.

Conversely, it also follows from Lemma \ref{lm1}, Proposition \ref{p2}, Theorem \ref{th1} and Proposition \ref{p3}.
\end{proof}

\begin{Remark}
Following the notations in the proof of Theorem \ref{th2}, we have
the following results:
\begin{enumerate}
\item[(1)] If $d(x_{1})$ and $d(x_{2})$ need not all $\bar{1}$,
then there exists a nondegenerate super-symmetric (skew
super-symmetric) invariant bilinear form on
$\mathrm{Ind}_{K}(V_{\sigma})$ if and only if there exists an
isomorphism of $L_{0}$-modules $\zeta:V\rightarrow (V_{\sigma})^{*}$
such that
\[\zeta(v_{1})(v_{2})=(-1)^{d(v_{1})d(v_{2})}\zeta(v_{2})(v_{1})\quad
(\zeta(v_{1})(v_{2})=-(-1)^{d(v_{1})d(v_{2})}\zeta(v_{2})(v_{1})),\]
for all $v_{1},v_{2}\in V$.
\item[(2)] If $d(x_{1})=d(x_{2})=\bar{1}$, then there exists a
nondegenerate super-symmetric (skew super-symmetric) invariant
bilinear form on $\mathrm{Ind}_{K}(V_{\sigma})$ if and only if there
exists an isomorphism of $L_{0}$-modules $\zeta:V\rightarrow
(V_{\sigma})^{*}$ such that
$\zeta(v_{1})(v_{2})=-(-1)^{d(v_{1})d(v_{2})}\zeta(v_{2})(v_{1})$
($\zeta(v_{1})(v_{2})=(-1)^{d(v_{1})d(v_{2})}\zeta(v_{2})(v_{1})$),
for all $v_{1},v_{2}\in V$.
\end{enumerate}
\end{Remark}

\section{Generalized reduced Verma modules and mixed products of modules}
\label{}
In this section, the relation between generalized reduced Verma modules and mixed products of modules
over $\mathbb{Z}$-graded modular Lie superalgebras of Cartan type will be discussed.

\begin{Proposition}\label{th3}
Let $L$ be a $\mathbb{Z}$-graded Lie superalgebra over $\mathbb{F}$
and $V=\oplus_{i\geq 0}V_{i}$ be a positively and transitively
graded $L$-module such that $z_{i}\cdot V=0$, $1\leq i\leq k$. Then
the linear mapping $\psi:V\rightarrow \mathrm{Coind}_{K}(V_{0})$
defined by $\psi(v)(x)=(-1)^{d(x)d(v)}\mathrm{pr_{0}}(x\cdot v)$,
for all $x\in U(L)$ and $v\in V$, is an injective homomorphism of
$L$-modules, where $\mathrm{pr_{0}}:V\rightarrow V_{0}$ denotes the
canonical projection. In particular,
$\psi(V_{0})=\mathrm{Coind}_{K}(V_{0})_{0}$ and $zd(\psi)=0$.
\end{Proposition}

\begin{proof}
Note that $\mathrm{pr_{0}}$ is a homomorphism of
$\theta(L,K)$-modules. In fact, for any $h_{j}\in \theta(L,K)_{j}$
and $v_{i}\in V_{i}$, we have $\mathrm{pr_{0}}(h_{j}\cdot
v_{i})=(-1)^{d(h_{j})d(\mathrm{pr_{0}})}h_{j}\cdot\mathrm{pr_{0}}(v_{i})$,
where $i,j\in \mathbb{N}_{0}$. Since the mapping $U(L)\rightarrow V$
defined by $x\mapsto (-1)^{d(x)d(v)}x\cdot v$ also satisfies this
property, $\psi$ is well-defined. Moreover, for an arbitrary element
$l\in L$, we obtain
\begin{eqnarray*}
\psi(l\cdot v)(x)&=&(-1)^{d(x)(d(l)+d(v))}\mathrm{pr_{0}}(x\cdot(l\cdot v))\\
&=&(-1)^{d(l)(d(x)+d(v))}\psi(v)(x\cdot l)=(-1)^{d(l)d(\psi)}(l\cdot \psi(v))(x).
\end{eqnarray*}
Therefore, $\psi$ is a homomorphism of $L$-modules. To prove $\psi$
is injective, we assume that $\mathrm{ker}\psi\neq 0$. Evidently,
$zd(\psi)=0$ and thereby $\mathrm{ker}\psi$ is a
$\mathbb{Z}$-homogeneous subspace of $V$. Then $\mathrm{ker}\psi\neq
0$ leads to the existence of a minimal $i\geq 0$ such that
$\mathrm{ker}\psi\cap V_{i}\neq 0$. Let $v_{i}\in
\mathrm{ker}\psi\cap V_{i}$ and $l\in L_{-j}$ ($j>0$). It follows
that $x\cdot v_{i}=\mathrm{pr_{0}}(x\cdot
v_{i})=(-1)^{d(x)d(v_{i})}\psi(v_{i})(x)=0$ for every $x\in
U(L)_{-i}$. If $q\neq j-i$, then
\[
\psi(l\cdot v_{i})(x)=(-1)^{d(x)(d(l)+d(v_{i}))}\mathrm{pr_{0}}(x\cdot(l\cdot v_{i}))=0,
\]
where $x\in U(L)_{q}$. If $q=j-i$, then $xl\in U(L)_{-i}$ and
$(xl)\cdot v_{i}=0$. Consequently, $l\cdot v_{i}$ belongs to the
trivial subspace $\mathrm{ker}\psi\cap V_{i-j}$. Since $V$ is
transitive, $v_{i}\in V_{0}$ and $i=0$. As a result, $x\cdot
v_{0}=0$ for all $x\in U(L)_{0}$. It follows from $1\in U(L)_{0}$
that $v_{0}=0$. This conclusion confutes the assumption
$\mathrm{ker}\psi\neq 0$ and thereby $\psi$ is an injective
homomorphism of $L$-modules.

Let $\mu:\mathrm{Coind}_{K}(V_{0})_{0}\rightarrow V_{0}$ such that
$\mu(f)=f(1)$. Let $x$ be an element of $U(L)_{j}$. If $j\neq 0$,
then $\mathrm{pr_{0}}(x\cdot f(1))=0$ and $f(x)=0$. In the case of
$j=0$, the P-B-W theorem provides a presentation
$x=\sum_{j=1}^{n}\sum_{i\geq 0}a_{ij}b_{ij}$, where $a_{ij}\in
U(K)_{i}$ and $b_{ij}\in U(L^{-})_{-i}$. Clearly,
\begin{eqnarray*}
& &f(x)-(-1)^{d(x)d(f)}\mathrm{pr_{0}}(x\cdot f(1))\\
&=&\sum_{j=1}^{n}\sum_{i\geq 0}((-1)^{d(a_{ij})d(f)}a_{ij}f(b_{ij})
-(-1)^{d(x)d(f)}a_{ij}\mathrm{pr_{0}}(b_{ij}f(1)))\\
&=&\sum_{j=1}^{n}((-1)^{d(a_{0j})d(f)}a_{0j}f(b_{0j})-(-1)^{d(x)d(f)}a_{0j}\mathrm{pr_{0}}(b_{0j}f(1)))\\
&=&\sum_{j=1}^{n}(-1)^{d(x)d(f)}(a_{0j}b_{0j}f(1)-a_{0j}b_{0j}f(1))=0.
\end{eqnarray*}
For an arbitrary element $x\in U(L)$,
$f(x)=(-1)^{d(x)d(f)}\mathrm{pr_{0}}(x\cdot f(1))$. Consequently,
$\psi\circ\mu=\mathrm{id}_{\mathrm{Coind}_{K}(V_{0})_{0}}$ and
$\psi(V_{0})=\mathrm{Coind}_{K}(V_{0})_{0}$.
\end{proof}

For $\alpha=(\alpha_1,\ldots,\alpha_k)\in\mathbb{N}_0^k$, we put
$|\alpha|:=\sum_{i=1}^{k}\alpha_{i}$. Let $\mathcal
{O}(k,\underline{m})$ denote the divided power algebra over
$\mathbb{F}$ with a $\mathbb{F}$-basis
$\{x^{(\alpha)}\mid\alpha\in\mathbb{A}(k,\underline{m})\}$, where
\[
\mathbb{A}(k,\underline{m}):=\left\{\alpha:=(\alpha_{1},\ldots,\alpha_{k})\in\mathbb{N}_0^{k}\mid
0\leqslant\alpha_{i}\leqslant p^{m_{i}}-1, i=1,2,\ldots,k\right\}.
\]
Let $\Lambda(l)$ be the exterior superalgebra over $\mathbb{F}$ in
$l$ variables $\xi_{1},\xi_{2}$, $\ldots,\xi_{l}$. Denote by
$\mathcal{O}(k,l,\underline{m})$ the tenser product
$\mathcal{O}(k,\underline{m})\otimes_{\mathbb{F}}\Lambda(l)$.

Put $\mathrm{Y}_{0}:=\{1,2,\ldots,k\}$ and
$\mathrm{Y}_{1}:=\{1,2,\ldots,l\}$. If $u\in\mathbb{B}_{s}$, $j\in
\{u\}$, then we suppose that $u-\langle j\rangle\in
\mathbb{B}_{s-1}$ such that $\{u-\langle
j\rangle\}=\{u\}\setminus\{j\}$. Let $u(j)=|\{l\in\{u\}\mid l<j\}|$.
If $j\in \mathrm{Y}_{1}\setminus\{u\}$, then we put $u(j)=0$ and
$\xi^{u-\langle j\rangle}=0$. Clearly,
$\left\{x^{(\alpha)}\xi^{u}\mid
\alpha\in\mathbb{A}(k,\underline{m}), u\in \mathbb{B}\right\}$
constitutes an $\mathbb{F}$-basis of
$\mathcal{O}(k,l,\underline{m})$ and
$zd(x^{(\alpha)}\xi^{u})=|\alpha|+|u|\geq0$.

Let $D_{1},\ldots,D_{k},d_{1},\ldots,d_{l}$ be the linear
transformations of $\mathcal{O}(k,l,\underline{m})$ and
$\varepsilon_i:=(\delta_{i1}$, $\ldots,\delta_{ik})$ such that
\begin{eqnarray*}
   D_i(x^{(\alpha)}\xi^{u})=x^{(\alpha-\varepsilon_{i})}\xi^{u},\quad
   i\in\mathrm{Y}_{0},\quad
   d_j(x^{(\alpha)}\xi^{u})=(-1)^{u(j)}x^{(\alpha)}\xi^{u-\langle j\rangle}, \quad  j\in\mathrm{Y}_{1},
\end{eqnarray*} where $\delta_{ij}$ is Kronecker delta, defined by
$\delta_{ij}=1$ if $i=j$ and $\delta_{ij}=0$ otherwise.

Modular Lie superalgebras of Cartan type $L(k,l,\underline{m})$
($L=W, S, H, K$) are subalgebras of the derivation superalgebras of
$\mathcal{O}(k,l,\underline{m})$. For the precise definitions please
refer to \cite{zhang}. If $L=W, S, H$, then $\{D_{1},\ldots,D_{k}\}$
is the canonical basis of $L(k,l,\underline{m})^{-}\cap
L(k,l,\underline{m})_{\bar{0}}$ and $\{d_{1},\ldots,d_{l}\}$ is the
canonical basis of $L(k,l,\underline{m})^{-}\cap
L(k,l,\underline{m})_{\bar{1}}$. The definition of the product in
$L(k,l,\underline{m})$ (see \cite{zhang}) entails the vanishing
$\mathrm{ad}D_{i}^{p^{m_{i}}}$ on $L(k,l,\underline{m})$ and we
therefore define $z_{i}:=D_{i}^{p^{m_{i}}}$, $1\leq i\leq k$.

\begin{Theorem}\label{c}
Let $L(k,l,\underline{m})$ ($L=W, S, H$) denote a $\mathbb{Z}$-graded Lie superalgebra of Cartan type.
If $V$ is an $L(k,l,\underline{m})_{0}$-module,
then $\mathrm{Ind}_{K}(V_{\sigma})$ is isomorphic to the mixed product $\mathcal{O}(k,l,\underline{m})\otimes V$.
\end{Theorem}

\begin{proof}
Since $(\mathcal{O}(k,l,\underline{m})\otimes V)_{k}:=\langle a\otimes v\mid a\in \mathcal{O}(k,l,\underline{m})_{k}, v\in V\rangle$,
the mixed product is a positively graded module.
According to the definition of the mixed product (see \cite{zhang1}), we have
\begin{eqnarray*}
D_{i}(x^{(\alpha)}\xi^{u}\otimes
v)&=&x^{(\alpha-\varepsilon_{i})}\xi^{u}\otimes v,\quad
i\in\mathrm{Y}_{0},\\
d_{j}(x^{(\alpha)}\xi^{u}\otimes
v)&=&(-1)^{u(j)}x^{(\alpha)}\xi^{u-\langle j\rangle}\otimes v,\quad
j\in\mathrm{Y}_{1}, \end{eqnarray*} where
$\alpha\in\mathbb{A}(k,\underline{m})$, $u\in \mathbb{B}$ and $v\in
V$. The first equality shows that
$z_{i}(\mathcal{O}(k,l,\underline{m})\otimes V)=0$, $1\leq i\leq k$.
The above equalities also ensure the transitivity of
$\mathcal{O}(k,l,\underline{m})\otimes V$. Proposition \ref{th3}
furnishes an embedding from $\mathcal{O}(k,l,\underline{m})\otimes
V$ into $\mathrm{Coind}_{K}(V)$. Since
\[\mathrm{dim}(\mathrm{Coind}_{K}(V))=\mathrm{dim}(\mathcal{O}(k,l,\underline{m})\otimes
V)=2^{l}p^{m_{1}+\cdots+m_{k}}\mathrm{dim}V,\] the mapping is
bijective. Then Lemma \ref{lm1} gives an isomorphism between
$\mathrm{Ind}_{K}(V_{\sigma})$ and
$\mathcal{O}(k,l,\underline{m})\otimes V$.
\end{proof}

\begin{Remark}
Let notations be as in Theorem \ref{th2} and \ref{c}. Then the
following statements are equivalent.
\begin{enumerate}
\item[(1)] There exists a nondegenerate super-symmetric or
skew super-symmetric invariant bilinear form on the mixed product
$\mathcal{O}(k,l,\underline{m})\otimes V$.
\item[(2)] There exists an
isomorphism of $L(k,l,\underline{m})_{0}$-modules
$\zeta:V\rightarrow (V_{\sigma})^{*}$ such that
$\zeta(v)(w)=(-1)^{d(v)d(w)}\zeta(w)(v)$ or
$\zeta(v)(w)=-(-1)^{d(v)d(w)}\zeta(w)(v)$, for all $v,w\in V$.
\end{enumerate}
\end{Remark}

\section*{Acknowledgments}
This work was supported by the NNSF of China (Grant No.11171055),
Natural Science Foundation of Jilin province (No. 20130101068) and
the Fundamental Research Funds for the Central Universities
(No.12SSXT139). The authors thank professors Liangyun Chen, Baolin
Guan, Li Ren for their helpful comments and suggestions.

\bibliographystyle{model1a-num-names}

\end{document}